\documentclass{article}%
\usepackage{amsfonts}
\usepackage{amsmath}
\usepackage[a4paper]{geometry}
\usepackage{amssymb}
\usepackage{graphicx}%
\providecommand{\U}[1]{\protect\rule{.1in}{.1in}}
\newtheorem{theorem}{Theorem}
\newtheorem{acknowledgement}[theorem]{Acknowledgement}

\newtheorem{corollary}[theorem]{Corollary}

\newtheorem{definition}[theorem]{Definition}

\newtheorem{lemma}[theorem]{Lemma}

\newtheorem{proposition}[theorem]{Proposition}
\newtheorem{remark}[theorem]{Remark}

\newenvironment{proof}[1][Proof]{\noindent\textbf{#1.} }{\ \rule{0.5em}{0.5em}}
\begin{document}

\title{The differentablity of horizons along their generators}
\author{D. Szeghy\thanks{Email:szeghy@caesar.elte.hu}\\E\"{o}tv\"{o}s University, Hungary, 1117 Budapest, P\'{a}zm\'{a}ny P\'{e}ter
stny. 1/C}
\maketitle

\begin{abstract}
Let $H$ be a (past directed) horizon in a time-oriented Lorentz manifold and
$\gamma:[\left(  \alpha,\beta\right)  \rightarrow H$ a past directed generator
of the horizon, where $[\left(  \alpha,\beta\right)  $ is $[\alpha,\beta)$ or
$\left(  \alpha,\beta\right)  $. It is proved that either at every point of
$\gamma\left(  t\right)  ,~t\in\left(  \alpha,\beta\right)  $ the
differentiability order of $H$ is the same, or there is a so-called
differentiability jumping point $\gamma\left(  t_{0}\right)  ,~t_{0}\in\left(
\alpha,\beta\right)  $ such that $H$ is only differentiable at every point
$\gamma\left(  t\right)  ,~t\in\left(  \alpha,t_{0}\right)  $ but not of class
$C^{1}$ and $H$ is exactly of class $C^{1}$ at every point $\gamma\left(
t\right)  ,~t\in\left(  t_{0},\beta\right)  $. We will use in the proof a
result which shows, that every mathematical horizon in the sense of P. T.
Chru\'{s}ciel locally coincides with a Cauchy horizon.

MSC 53C50

\end{abstract}

\section{Introduction}

There are different objects in physics which are called horizons. A
mathematical definition of horizons was given by P. T. Chru\'{s}ciel in
\cite{C}, see definition \ref{Dhorizon copy(1)} in this paper. This definition
includes e.g. black hole event horizons and Cauchy horizons as well. It turned
out that the set of those points at which the horizon is not differentiable
can not be neglected, for references see e.g. the introduction of \cite{Sz}.
First J. K. Beem and A. Kr\'{o}lak in \cite{B-K} proved that at the interior
points of a generator the horizon is differentiable. Moreover, at the possible
endpoint $p$ of a generator, $H$ is differentiable if and only if $N\left(
p\right)  =1$, i.e. $p$ is the endpoint of only $1$ generator. Later
Chru\'{s}ciel in \cite{C} gave a simpler proof which works for mathematical
horizons. The higher order differentiability properties of the horizon was
studied by the author in \cite{Sz}, it was proved that along a generator the
differentiability order of the horizon at the generator points can change only
once, see theorem \ref{structural} in this paper . If a change happens, then
from simple differentiability it changes to a higher order $C^{k}$, $k\geq1$
differentiability. Our goal is to show that $k=1$. In \cite{Sz} a geometric
characterization was also given to describe the point at which the
differentiability changes with help of non-injectivity points.

First we recall the definitions, known results and some notations. Then give
the proof in the Minkowski $3$-dimensional space-time in which the idea is
easy to follow, then we prove the general case. Some technical details are
presented in the appendix.

\section{Definitions, known results}

We will use the definitions and notations from the book of J. K. Beem et all
\cite{B-E_E} e.g. for the chronological (causal), past (future), causality
relations and their standard notations.

We use the definition in \cite{C}, to define a horizon in the mathematical
sense. Let $\left(  M,g\right)  $ denote a time-oriented Lorentz manifold
throughout the paper.

\begin{definition}
\label{Dhorizon copy(1)}A subset $H\subset M$ is called a \textbf{horizon}
if:\newline$\left(  1\right)  $ it is a topological hypersurface;\newline%
$\left(  2\right)  $ it is achronal;\newline$\left(  3\right)  $ for every
point $p\in H$ there is a past inextendable, past directed, light-like
geodesic $\gamma:[0,\alpha)\rightarrow H$, such that $\gamma\left(  0\right)
=p$.
\end{definition}

An achronal topological hypersurface can be considered as a continuous
function graph, for a precise introduction we refer the reader to \cite{Sz}.
Note that in $\left(  2\right)  $ we could use local achronality, even in that
case the results remain true.

For the sake of simplicity the notation $[\left(  \alpha,\beta\right)  $ will
mean the interval $\left(  \alpha,\beta\right)  $ or the interval
$[\alpha,\beta)$, where $\alpha,\beta\in\mathbb{R}\cup\left\{  -\infty
,\infty\right\}  $

\begin{definition}
\label{Dgener copy(1)}Let $\gamma:[(\alpha,\beta)\rightarrow H$ be a past
directed, past inextendable, light-like geodesic, such that it is future
inextendable on the horizon, i.e. there is no $\varepsilon>0$, such that a
light-like geodesic $\widetilde{\gamma}:\left(  \alpha-\varepsilon
,\beta\right)  \rightarrow H$ exists, for which $\widetilde{\gamma}\left(
t\right)  =\gamma\left(  t\right)  $ for every $t\in\lbrack(\alpha,\beta)$.
Then $\gamma$ is called a \textbf{generator}.
\end{definition}

Thus, a generator can be defined on an open interval $\left(  \alpha
,\beta\right)  $, or on a closed-open one $[\alpha,\beta)$, depending only on
the generator.

Let $N\left(  p\right)  $ denote number of the generators at $p\in H$. The
achronality of $H$ gives that generator cannot intersect or have branching
points exept at their possible common endpoint\footnote{"Endpoint" in the
sense of chronological relations on $\gamma$. By us, the generators are past
directed and the chronological endpoint of $\gamma:\left[  \alpha
,\beta\right)  \rightarrow H$ is $\gamma\left(  \alpha\right)  $ which is the
"starting point" in the sense of parametrization.}, see e.g. \cite{Sz}. Thus
$N\left(  p\right)  =1$ must hold if $p$ is an interior point of of a generator.

\begin{proposition}
\label{Prop_diff=c1}The following are equivalent:

\begin{itemize}
\item $H$ is differentiable at every point of an open set $U\subset H$;

\item $\forall p\in U\subset H,~N\left(  p\right)  =1$;

\item every point $p\in U$ is an interior point of a generator;

\item $H$ is differentiable at least of class $C^{1}$ at every point $p\in U$.
\end{itemize}
\end{proposition}

\begin{proof}
See, proposition 3.3. in \cite{B-K}, where the proof also works for
mathematical horizons.
\end{proof}

We recall shortly what it means that $H$ is differentiable of class $C^{k}$ at
a point $p\in H$, the precise description can be found in \cite{Sz} \ at the
beginning of the section: \textit{The differentiability order of horizons}.
Let $N$ be a topological submanifold in the smooth manifold $M$. It is
differentiable of class $C^{k}$ at $p\in N$ if there is a suitable smooth
coordinate chart in a neighbourhood of $p\in M$ such that the coordinate
functions of $N$ are $k$ times differentiable at $p$ and the $k$-th
derivatives are continuous at $p.$ It means that in a neighbourhood of $p$ the
$\left(  k-1\right)  $th derivatives of the coordinate functions are
continuous, but the continuity of the $k$-th derivatives are guaranteed only
at $p$. The differentiability class is exactly of class $C^{k}$ if it is off
class $C^{k}$ but not of class $C^{k+1}$.

\begin{theorem}
\bigskip\label{structural}For every generator $\gamma:[\left(  \alpha
,\beta\right)  \rightarrow H$ a unique parameter $t_{0}\in\lbrack\alpha
,\beta]$ exists, where $\alpha,~\beta\in\mathbb{R\cup}\left\{  -\infty
,\infty\right\}  $, such that there is a $k\geq1$ for which\newline$\left(
1\right)  $ $H$ is exactly of class $C^{k}$ at every $\gamma\left(  t\right)
$, for which $t>t_{0};$\newline$\left(  2\right)  $ $H$ is differentiable, but
not of class $C^{1}$ at every $\gamma\left(  t\right)  $, for which
$t\in(\alpha,t_{0}];$\newline$\left(  3\right)  $ $H$ is differentiable at the
possible endpoint $\gamma\left(  \alpha\right)  $ if and only if $N\left(
\gamma\left(  \alpha\right)  \right)  =1$.
\end{theorem}

\begin{proof}
See, in \cite{Sz} Theorem (Structure of the generators).
\end{proof}

\begin{definition}
Let $\gamma:[\left(  \alpha,\beta\right)  \rightarrow H$ be a generator and
assume that in the above theorem \ref{structural} defined point $t_{0}$ is not
equal to $\alpha$ or $\beta$. Then $\gamma\left(  t_{0}\right)  $ is called
the \textbf{differentiability jumping point of }$\gamma$.
\end{definition}

Now the main goal of this paper is to sharpen the above theorem, thus

\begin{theorem}
\label{Tmain}If on a generator $\gamma:[\left(  \alpha,\beta\right)
\rightarrow H$ the differentiability jumping point exists i.e. $t_{0}%
\in\left(  \alpha,\beta\right)  $, then in theorem \ref{structural}, $k=1$.
\end{theorem}

An example in \cite{Sz} is given where the differentiability jumping point exists.

For the proof we will need the definition of Cauchy horizons, see p. 419
definition 35 in \cite{O'N}

\begin{definition}
\label{defCouchy}Let $S$ be an achronal subset in a time oriented Lorentz
manifold $\left(  M,g\right)  $. The \textbf{future Cauchy development} of $S$
is%
\[
D^{+}\left(  S\right)  \overset{def}{=}\left\{  z\in M\,|\,\text{every past
directed inextendable causal curve from }z\text{ intersects }S\right\}  ;
\]
and the \textbf{future Cauchy horizon} of $S$ is%
\[
H^{+}\left(  S\right)  \overset{def}{=}\partial D^{+}\left(  S\right)
-\overline{S}.
\]

\end{definition}

It is known that $H^{+}\left(  S\right)  $ is a horizon in the Lorentz
manifold $\left(  M\backslash\overline{S},g\right)  $, where $\overline{S}$ is
the closure of $S$.

The main idea of the proof of theorem \ref{Tmain} is presented first in a
Minkowski $3$-space in case of a Cauchy horizon because it is easier to
understand the idea of the proof in this particular case. Then we will show
how to use this idea in the general case.%

\begin{figure}[ptb]%
\centering
\includegraphics[
height=3.2456in,
width=5.2217in
]%
{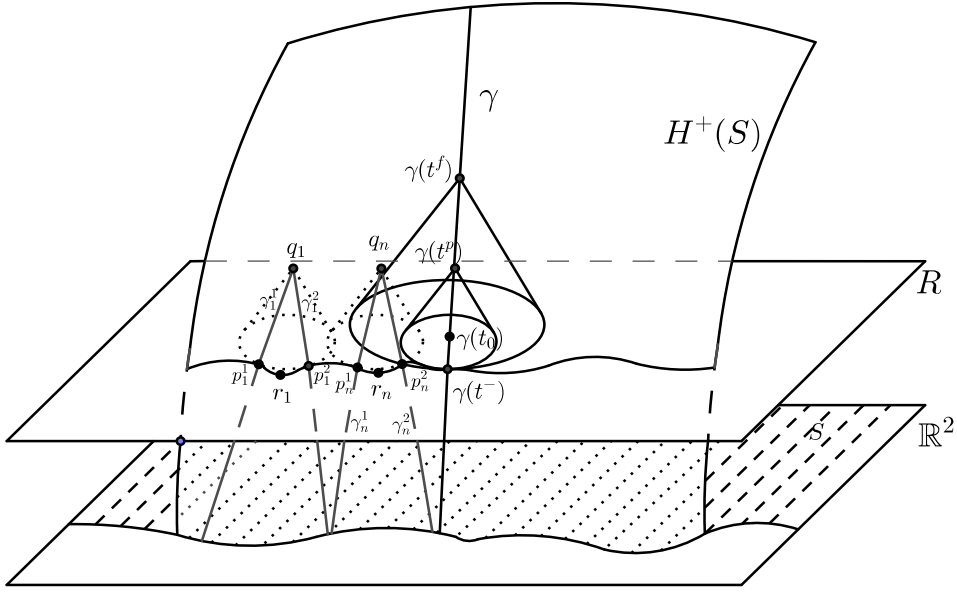}%
\caption{The points $r_{n}\in H^{+}\left(  S\right)  \cap R$ with big
curvature converge to $\gamma\left(  t^{-}\right)  $. }%
\label{F1}%
\end{figure}

Let $\left(  \mathbb{M}^{3},g_{M}\right)  $ be the standard Minkowski space
with the orthonormal base $e_{1},~e_{2},~e_{3},$ where $e_{1},~e_{2}$ are
space-like and $e_{3}$ is time-like. Consider $\mathbb{R}^{2}$ the space-like
linear subspace spanned by $e_{1}$ and $e_{2}$. Let $S\subset\mathbb{R}^{2}$
be an open set. Let $\gamma:$ $[\left(  \alpha,\beta\right)  \rightarrow
H^{+}\left(  S\right)  $ be a past directed generator in $H^{+}\left(
S\right)  $ and assume that $\gamma\left(  t_{0}\right)  $ is a
differentiability jumping point, see \textsc{figure \ref{F1}}. From theorem
\ref{structural} we know that $H^{+}\left(  S\right)  $ is only differentiable
at the points of $\gamma|_{\left(  \alpha,t_{0}\right]  }$ and it is of class
$C^{k},~k\geq1$ at the points of $\gamma|_{\left(  t_{0},\beta\right)  }$. By
an indirect assumption, assume that $k\geq2$. Let $t^{-}>t_{0}$ and $R$ an
affine subspace at $\gamma\left(  t^{-}\right)  $, which is parallel to
$\mathbb{R}^{2}$, see \textsc{figure \ref{F1}}. Since $H^{+}\left(  S\right)
$ is at least $C^{2}$ at $\gamma\left(  t^{-}\right)  $ therefore the curve
$\theta\overset{def}{=}R\cap H^{+}\left(  S\right)  $ is at least $C^{2}$ at
$\gamma\left(  t^{-}\right)  $. Let $t^{f}<t^{p}<t_{0}$, then $\partial\left(
J^{-}\left(  \gamma\left(  t^{f}\right)  \right)  \cap R\right)
,~\partial\left(  J^{-}\left(  \gamma\left(  t^{p}\right)  \right)  \cap
R\right)  $ are circles on $R$ touching $\theta$ at $\gamma\left(
t^{-}\right)  $ with curvatures $0<\kappa^{f}$ $<\kappa^{p}$ where the
curvature is the common curvature for $\mathbb{R}^{2}$ curves. Since
$H^{+}\left(  S\right)  $ is only differentiable, but not of class $C^{1}$ at
$\gamma\left(  t^{p}\right)  $, by proposition \ref{Prop_diff=c1}, there must
be a sequence $q_{n}\in H^{+}\left(  S\right)  $, $q_{n}\rightarrow
\gamma\left(  t^{p}\right)  $ with $N\left(  q_{n}\right)  \geq2$. Thus, there
are generators $\gamma_{n}^{1},~\gamma_{n}^{2}$ starting at $q_{n}$ which
intersects $\theta$ at $p_{n}^{1},~p_{n}^{2}$. As $q_{n}\rightarrow
\gamma\left(  t^{p}\right)  $ we have that the curvature $\kappa_{n}$ of the
circles $\partial\left(  J^{-}\left(  q_{n}\right)  \cap R\right)  $ converges
to $\kappa^{p}$. Since $\left(  I^{-}\left(  q_{n}\right)  \cap R\right)
\subset D^{+}\left(  S\right)  $ and $p_{n}^{1},~p_{n}^{2}\in\theta
\subset\partial D^{+}\left(  S\right)  $ we have that $\theta$ and the circles
$\partial\left(  J^{-}\left(  q_{n}\right)  \cap R\right)  $ are tangent at
$p_{n}^{1}$ and at $p_{n}^{2}$. The curve $\theta$ is differentiable (at
least) twice in a neighbourhood of $\gamma\left(  t^{-}\right)  $, however we
know only that the second derivatives are continuous $\gamma\left(
t^{-}\right)  $. Nevertheless we can calculate formally the curvature of
$\theta$ in the neighbourhood of $\gamma\left(  t^{-}\right)  $, see
definition \ref{Dref} in the appendix. \ Since the limit of generators is a
generator and $N\left(  \gamma\left(  t^{p}\right)  \right)  =1$ we have that
$\gamma_{n}^{i}\rightarrow\gamma|_{\left[  t^{-},\beta\right)  },~i=1,2$. This
follows that $p_{n}^{1},~p_{n}^{2}\rightarrow\gamma\left(  t^{-}\right)  $.
Therefore, we can apply lemma \ref{LAgetcurvature} to the segment of $\theta$
between $p_{n}^{1}$ and $p_{n}^{2}$ and to the segment of the circle between
$p_{n}^{1}$ and $p_{n}^{2}$. That lemma gives a point $r_{n}$ on the segment
$\theta|_{p_{n}^{1},p_{n}^{2}}$ at which the curvature $\kappa_{\theta}\left(
r_{n}\right)  $ of $\theta$ is at least the curvature $\kappa_{n}$ of the
circle $\partial\left(  J^{-}\left(  q_{n}\right)  \cap R\right)  $. As
$\theta$ is $C^{2}$ at $\gamma\left(  t^{-}\right)  $ we have that
$\kappa_{\theta}\left(  r_{n}\right)  \rightarrow\kappa_{\theta}\left(
\gamma\left(  t^{-}\right)  \right)  $, where $\kappa_{\theta}\left(
\gamma\left(  t^{-}\right)  \right)  $ is the curvature of $\theta$ at
$\gamma\left(  t^{-}\right)  $. Now, since $\kappa_{\theta}\left(
r_{n}\right)  \geq\kappa_{n}$ and $\kappa_{n}\rightarrow\kappa^{p}$ we have
that
\[
\kappa_{\theta}\left(  \gamma\left(  t^{-}\right)  \right)  \geq\kappa
^{p}>\kappa^{f}.
\]
By remark \ref{LAbigcurvature} we have that in this case $\theta$ must
intersect the interior of the disc $J^{-}\left(  \gamma\left(  t^{p}\right)
\right)  \cap R$, which is in $I^{-}\left(  \gamma\left(  t^{p}\right)
\right)  \cap R.$ Thus there is a point $x\in\theta\subset H^{+}\left(
S\right)  \cap R$ in $I^{-}\left(  \gamma\left(  t^{p}\right)  \right)  \cap
R$ which contradicts the achronality of $H^{+}\left(  S\right)  $, since
$x$,$~\gamma\left(  t^{f}\right)  \in H^{+}\left(  S\right)  $ can not be
chronologically related.

In the general case we use a similar proof where the steps of the proof are
the following. We will use the above defined notations in these steps.

\begin{enumerate}
\item Since we examine the horizon locally at the differentiability jumping
point, we can chose a suitably small neighbourhood $U$ of $\gamma\left(
t_{0}\right)  $ and a coordinate system on it, such that the restricted
smaller Lorentz manifold $\left(  U.g|_{U}\right)  $ can be considered as a
ball $\mathbb{B}^{n}\subset\mathbb{R}^{n}$, endowed with a Lorentz metric $g $
which is suitably close to a Minkowski metric.

\item There is a space-like affine hyperplane $\mathbb{R}^{n-1}\subset
\mathbb{B}^{n}$ and a set $S\subset\mathbb{R}^{n-1}$ for which the horizon at
$\gamma\left(  t_{0}\right)  $ locally is the same as the Cauchy horizon
$H^{+}\left(  S\right)  $, see lemma \ref{Lmodel}. Thus it is enough to prove
the theorem for Cauchy horizons, i.e. we will assume that $H=H^{+}\left(
S\right)  $.

\item Therefore the intersections $\partial\left(  J^{-}\left(  \gamma\left(
t^{f}\right)  \right)  \cap R\right)  ,\dots$ are not circles, because of the
higher dimension, nor spheres, but can be arbitrary close to spheres, if $U$
was small enough.

\item Let $P\subset R\simeq\mathbb{R}^{n-1}$ be a $2$-dimensional normal plane
of $\partial\left(  J^{-}\left(  \gamma\left(  t^{f}\right)  \right)  \cap
R\right)  $ at $\gamma\left(  t^{-}\right)  $, i.e. $\gamma\left(
t^{-}\right)  \in P$ and $P$ is parallel to the normal\footnote{We will fix an
arbitrary Euclidean inner product on $R\approx\mathbb{R}^{n-1}$ to be able to
speak about normal directions.} vector of $\partial\left(  J^{-}\left(
\gamma\left(  t^{f}\right)  \right)  \cap R\right)  $ at $\gamma\left(
t^{-}\right)  $. The curve $\partial\left(  J^{-}\left(  \gamma\left(
t^{f}\right)  \right)  \cap R\right)  \cap P$ is called a normal curve of
$\partial\left(  J^{-}\left(  \gamma\left(  t^{f}\right)  \right)  \cap
R\right)  $ at $\gamma\left(  t^{-}\right)  $. If $\gamma\left(  t^{p}\right)
$ is much closer to $R$, i.e. to $\gamma\left(  t^{0}\right)  $, than
$\gamma\left(  t^{f}\right)  $ then there is a constant $\kappa$ such that the
curvature at $\gamma\left(  t^{-}\right)  $ of any\ normal curve of
$\partial\left(  J^{-}\left(  \gamma\left(  t^{f}\right)  \right)  \cap
R\right)  $ at $\gamma\left(  t^{-}\right)  $ is less than $\kappa$ and the
curvature at $\gamma\left(  t^{-}\right)  $ of any\ normal curve of
$\partial\left(  J^{-}\left(  \gamma\left(  t^{p}\right)  \right)  \cap
R\right)  $ at $\gamma\left(  t^{-}\right)  $ is much bigger that $\kappa+1$.

\item There is a suitable fixed point $o\in R$ for which the $2$-dimensional
planes $P_{n}\subset R\simeq\mathbb{R}^{n-1}$ defined by $O,~p_{n}^{1}%
,~p_{n}^{2}$ (i.e. $O,~p_{n}^{1},~p_{n}^{2}\in P_{n}$) will have a suitable
subsequent which will converge to a normal plane $P$ of $\partial\left(
J^{-}\left(  \gamma\left(  t^{f}\right)  \right)  \cap R\right)  $ at
$\gamma\left(  t^{-}\right)  $.

\item The intersections $\partial\left(  J^{-}\left(  q_{n}\right)  \cap
R\right)  \cap P_{n}$ are curves which are not circles but have curvatures
much bigger than $\kappa+1$ if $\gamma\left(  t^{p}\right)  $ is close enough
to $R$ and $q_{n}$ is close enough to $\gamma\left(  t^{p}\right)  $, i.e. $n$
is big enough.

\item The plane curves $\partial\left(  J^{-}\left(  q_{n}\right)  \cap
R\right)  \cap P_{n}$ and $H^{+}\left(  S\right)  \cap P_{n}$ are tangent at
$p_{n}^{1}$ and $p_{n}^{2}$, moreover we can apply lemma \ref{LAgetcurvature}
to get a point $r_{n}$ in $H^{+}\left(  S\right)  \cap P_{n}$ where the
curvature of the plane curve $H^{+}\left(  S\right)  \cap P_{n}$ is bigger
than $\kappa+1;$

\item By the convergence $P_{n}\rightarrow P$ and the continuity of the second
order derivatives at $\gamma\left(  t^{-}\right)  $ the plane curve $P\cap
H^{+}\left(  S\right)  $ must have bigger curvature as $\kappa+1$.

\item The ending is the same, since the curve $P\cap H^{+}\left(  S\right)  $
must have a point $x$ in the interior of $P\cap\left(  I^{-}\left(
\gamma\left(  t^{f}\right)  \right)  \cap R\right)  ,$ because the boundary
$P\cap\partial\left(  I^{-}\left(  \gamma\left(  t^{f}\right)  \right)  \cap
R\right)  =P\cap\partial\left(  J^{-}\left(  \gamma\left(  t^{f}\right)
\right)  \cap R\right)  $ has curvature at most $\kappa$ at $\gamma\left(
t^{-}\right)  $ and we will apply remark \ref{LAbigcurvature}. But this would
contradict the achronality of $H^{+}\left(  S\right)  $ since $x,\gamma\left(
t^{f}\right)  \in H^{+}\left(  S\right)  $.
\end{enumerate}

Now we prove the above idea step by step. For the sake of simplicity a
suitable subsequence of $q_{n}$ will be also denoted by $q_{n}$.

Since the statement in theorem \ref{Tmain} is a local property, we will use a
suitable local coordinate system. Let $p\in H$ be a point of the horizon and
$U\subset M$ an open geodetically convex neighbourhood of $p$. Moreover, let
$\varphi:U\rightarrow\mathbb{B}^{n}\subset\mathbb{R}^{n}$ be a smooth
diffeomorphism, like a coordinate chart. \ We can consider the Lorentz metric
on $\varphi\left(  U\right)  =\mathbb{B}^{n}$ defined by $\varphi$. \ If $U$
was small enough than the metric is close to a Minkowski metric. This was step
$1$.

Thus from hereon we will assume that \textbf{we have a Lorentz metric }%
$g$\textbf{\ on }$\mathbb{B}^{n}$\textbf{\ and a horizon }$H$\textbf{\ in
}$\left(  \mathbb{B}^{n},g\right)  $\textbf{.}

We will assume that $U$ is so small, that for every point $p\in U$ the causal
future (resp. past) of $p$ in $\left(  U,g|_{U}\right)  $ which is denoted by
$J^{+}\left(  U,p\right)  $ (resp. $J^{-}\left(  U,p\right)  $) is relatively
closed in $U$ and their interiors are $I^{+}\left(  U,p\right)  $ (resp.
$I^{-}\left(  U,p\right)  $).

Step $2$. in the proof of the main theorem is to show that every horizon
locally coincides with a local part of a Cauchy horizon:

\begin{lemma}
\label{Lmodel}Let $H$ be a horizon, then for every $p\in H$ there is an open
neighbourhood $U_{p}$ and a Cauchy horizon $H^{+}\left(  S\right)  $ such that
$H\cap U_{p}=H^{+}\left(  S\right)  \cap U_{p}$.
\end{lemma}

\begin{proof}
Remember that by our assumption we must prove that: \textit{if we have a
Lorentz metric }$g$\textit{\ on }$\mathbb{B}^{n}$\textit{\ and a horizon }%
$H$\textit{\ in }$\left(  \mathbb{B}^{n},g\right)  $\textit{, then for every
}$p\in H$\textit{\ there is an open neighbourhood }$U_{p}$\textit{\ and a
Cauchy horizon }$H^{+}\left(  S\right)  $\textit{\ such that }$H\cap
U_{p}=H^{+}\left(  S\right)  \cap U_{p}$\textit{. }

Now we will choose neighbourhoods $W_{p},~V_{p},~U_{p}$ of $p$ and an affine
hyperplane $F\subset\mathbb{B}^{n}$ suitably close to $p$ such that the
following hold:

\begin{enumerate}
\item $U_{p}\subset W_{p}\subset V_{p}$

\item $F\cap V_{p}$ is space-like

\item $\overline{U_{p}}\cap H$ is compact, where $\overline{U_{p}}$ is the
closure of $U_{p}$

\item for every point $q\in W_{p}$ every causal past directed and inextendable
curve from $q$ intersects $F\cap V_{p}$

\item $U_{p}\cap F=\emptyset$
\end{enumerate}

\emph{Note that for every point }$q\in\overline{U_{p}}\cap H$\emph{\ every
generator through }$q$\emph{\ also intersects }$F$\emph{\ }and the
intersection point is in the past of $q.$ Let
\[
S\overset{def}{=}\cup_{q\in\overline{U_{p}}\cap H}I^{-}\left(  q\right)  \cap
F,
\]
which is an open achronal set in $F$, thus, we can consider its Cauchy
development $D^{+}\left(  S\right)  $.

First we show that
\begin{equation}
\overline{U_{p}}\cap H\subset\overline{D^{+}\left(  S\right)  }. \label{Ee}%
\end{equation}
If $q\in\overline{U_{p}}\cap H$ and $q^{-}\in I^{-}\left(  q\right)  \cap
W_{p}$ then $J^{-}\left(  q^{-}\right)  \cap F\subset I^{-}\left(  q\right)
\cap F\subset S$ and by property $4.$we have that $q^{-}\in D^{+}\left(
S\right)  $. As $q^{-}$ can be arbitrary close to $q$ we have $q\in
\overline{D^{+}\left(  S\right)  }$.

Let $q\in\overline{U_{p}}\cap H$ and $\gamma_{q}$ a generator through $q$
(there can be more generators through $q$). Let $f_{\gamma_{q}}%
\overset{def}{=}F\cap\gamma_{q}$. Note that $\gamma_{q}$ can not "return" and
have more than one intersection point with $F$, if $U$ in step $1$ is small
enough. As $f_{\gamma_{q}}\in J^{-}\left(  q\right)  \cap F=\overline
{I^{-}\left(  q\right)  \cap F}$ we have that $I^{-}\left(  q\right)  \cap
F\subset S$ yields%
\begin{equation}
f_{\gamma_{q}}\in\overline{S}. \label{E0}%
\end{equation}

We show that $f_{\gamma_{q}}\notin S$. Assume on the contrary that
$f_{\gamma_{q}}\in S$. Since $\cup_{q\in\overline{U_{p}}\cap H}I^{-}\left(
q\right)  \cap F$ is an open cover of $S$ in $F$ there is a $q^{\ast}%
\in\overline{U_{p}}\cap H$ for which $f_{\gamma_{q}}\in I^{-}\left(  q^{\ast
}\right)  .$ But this would contradict the achronality of $H$ since both
$q^{\ast}$ and $f_{\gamma_{q}}\in\gamma_{q}$ are in $H$. Therefore,
$f_{\gamma_{q}}\notin S$ and by $\left(  \ref{E0}\right)  $ we have that%
\[
f_{\gamma_{q}}\in\partial_{F}S
\]
where the boundary is taken with respect to $F$. We will show that
\begin{equation}
q\in\partial D^{+}\left(  S\right)  . \label{E2}%
\end{equation}
This can be proved by contradiction. If $q\notin\partial D^{+}\left(
S\right)  $ then \ by $\left(  \ref{Ee}\right)  $ we have $q\in intD^{+}%
\left(  S\right)  $. Thus, there would be a point $q^{+}\in I^{+}\left(
q\right)  \cap D^{+}\left(  S\right)  \cap W_{p}$. Therefore $I^{-}\left(
q^{+}\right)  \supset J^{-}\left(  q\right)  $ would hold and imply $\left(
I^{-}\left(  q^{+}\right)  \cap F\right)  \supset\left(  J^{-}\left(
q\right)  \cap F\right)  \ni f_{\gamma_{q}}$. By property $4.$ and by the fact
that $S\subset F$, we have $f_{\gamma_{q}}\in\left(  I^{-}\left(
q^{+}\right)  \cap F\right)  \subset S$, this would contradict that
$f_{\gamma_{q}}\notin S$.

Since $\left(  \overline{U_{p}}\cap H\right)  \cap F=\emptyset$ by property
5., we have that
\[
\overline{U_{p}}\cap H\subset\partial D^{+}\left(  S\right)  -\overline{S}%
\]
which concludes the proof.
\end{proof}

Step $3.$ and $4.$ are put together in the following remark.

\begin{remark}
\label{Rcurvature}Let $\left(  \mathbb{B}^{n},g\right)  $ be a Lorentz
manifold and $V\subset\mathbb{B}^{n}$ a small open set. Let $J^{-}\left(
x,V\right)  $ be the causal past of $x$ in the restricted smaller Lorentz
manifold $\left(  V,g|_{V}\right)  $. Let $R$ be a hyperplane such that $V\cap
R$ is space-like and $x\in V$ a point for which $J^{-}\left(  x,V\right)  \cap
R\subset V$ and $J^{-}\left(  x,V\right)  \cap R\neq\emptyset$. \ Since $R$ is
a hyperplane it can be identified with $\mathbb{R}^{n-1}$. If $x$ is close
enough to $R$ then $\partial\left(  J^{-}\left(  x,V\right)  \cap R\right)  $
is smooth and diffeomorphic to a sphere. Therefore, if we equip $R=\mathbb{R}%
^{n-1}$ with a standard euclidean metric\footnote{We will fix a standart
metric on $R\approx\mathbb{R}^{n-1}$ and the curvatures, orthogonality, will
be taken with respect to this metric throughout this paper.} then at every
point $y\in\partial\left(  J^{-}\left(  x,V\right)  \cap R\right)  ~$we can
take the inward normal vector $n_{y}$. Let $v\in T_{y}\partial\left(
J^{-}\left(  x,V\right)  \cap R\right)  $ be a tangent vector at $y$ and $K$ a
$2$-dimensional plane through $y$ parallel to $n_{y}$ and $v$, i.e. a normal
plane at $y$ in the direction $v$. The intersection $\partial\left(
J^{-}\left(  x,V\right)  \cap R\right)  \cap K$ is a $2$-dimensional curve,
called normal curve in the direction of $v$, which has curvature
$\kappa\left(  y,v\right)  $ with respect to the normal vector $n_{y}$ thus
$\kappa\left(  y,v\right)  $ is the normal curvature of $\partial\left(
J^{-}\left(  x,V\right)  \cap R\right)  $ in the direction of $v$ with respect
to the invard normal direction $n_{y}$.\footnote{This is the classical normal
curvature in differential geometry. The sphere of radius $R$ will have
positive normal curvature $\frac{1}{R}$ in any tangential direction.} \ This
curvature depends on the identification $R\approx\mathbb{R}^{n-1}$ and the
euclidean metric on $\mathbb{R}^{n-1}$. Let
\[
\kappa_{x,R}^{\max}\left(  y\right)  \text{ and }\kappa_{x,R}^{\min}\left(
y\right)
\]
denote the maximum and minimum of the set $\left\{  \kappa\left(  y,v\right)
~|~v\in T_{p}\partial\left(  J^{-}\left(  x,V\right)  \cap F\right)  \right\}
$ of the normal curvatures at $y$. \ By the smoothness and the properties of
the exponential map it is clear that for every $0<\kappa$, if $x$ is close
enough to $F$ then $\kappa<\kappa_{x,R}^{\min}\left(  y\right)  $ for every
$y\in\partial\left(  J^{-}\left(  x,V\right)  \cap R\right)  $.
\end{remark}

Now we can proceed with the other steps, which will yield the proof of theorem
\ref{Tmain}.

\begin{proof}
[Proof of the Main theorem]By lemma \ref{Lmodel} we can assume that we have a
Lorentz manifold $\left(  \mathbb{B}^{n},g\right)  $ a space-like hyperplane
$F$ in it and a set $S\subset F$ such that $H^{+}\left(  S\right)  $ locally
coincides with our horizon at $\gamma\left(  t_{0}\right)  $. Thus it is
enough to prove the theorem for $H^{+}\left(  S\right)  $ at $\gamma\left(
t_{0}\right)  $.

The proof goes by contradiction. The situation is similar to \textsc{figure
\ref{F1}}, but $\mathbb{R}^{2}$ is replaced by $F$, the cones, spheres
(circles) are not real cones and spheres only smooth manifolds close to them.
Assume that $H^{+}\left(  S\right)  $ is only differentiable at every
$\gamma\left(  t\right)  ,~t\in\left(  \alpha,t_{0}\right)  $ and it is at
least of class $C^{2}$ at every $\gamma\left(  t\right)  ,~t\in\left(
t_{0},\beta\right)  $. Let $\gamma\left(  t^{f}\right)  ,~\gamma\left(
t^{p}\right)  ,~\gamma\left(  t^{-}\right)  \in\gamma\cap H^{+}\left(
S\right)  $ such that $t^{f}<t^{p}<t_{0}<t^{-}$ and all these points are close
enough to $t_{0}$. Moreover, let $R$ be a hyperplane parallel to $F$ which
intersects $\gamma$ at $\gamma\left(  t^{-}\right)  $. We can assume that
$\gamma\left(  t^{f}\right)  $ and $R$ are so close that $J^{-}\left(
x\right)  \cap F$ is a smooth embedded manifold and $\gamma\left(
t^{p}\right)  $ is "much closer" to $R$ than $\gamma\left(  t^{f}\right)  $
thus by remark \ref{Rcurvature} (steps $3$. and $4.$) the following hold. Let
$\kappa_{\gamma\left(  t^{f}\right)  ,F}^{\max}\left(  \gamma\left(
t^{-}\right)  \right)  $ be the maximal $2$-dimensional curvature of
$\partial\left(  J^{-}\left(  \gamma\left(  t^{f}\right)  \right)  \cap
R\right)  $ in $R\approx\mathbb{R}^{n-1}$ at $\gamma\left(  t^{-}\right)  $
and $\kappa_{\gamma\left(  t^{p}\right)  .R}^{\min}\left(  \gamma\left(
t^{-}\right)  \right)  $ be the minimal $2$-dimensional curvature of
$\partial\left(  J^{-}\left(  \gamma\left(  t^{p}\right)  \right)  \cap
R\right)  $ in $R\approx\mathbb{R}^{n-1}$ at $y,$ then

\begin{enumerate}
\item[P1] $\partial\left(  J^{-}\left(  \gamma\left(  t^{f}\right)  \right)
\cap R\right)  ,~\partial\left(  J^{-}\left(  \gamma\left(  t^{p}\right)
\right)  \cap R\right)  $ are smooth manifolds in $R\approx\mathbb{R}^{n-1}$;

\item[P2] $0<\kappa_{\gamma\left(  t^{f}\right)  ,R}^{\max}\left(
\gamma\left(  t^{-}\right)  \right)  <\kappa<\kappa+1<\kappa_{\gamma\left(
t^{p}\right)  .R}^{\min}\left(  \gamma\left(  t^{-}\right)  \right)  $ for
some $\kappa\in\mathbb{R}^{+}$;
\end{enumerate}

Since $H^{+}\left(  S\right)  $ is only differentiable at $\gamma\left(
t^{f}\right)  $ but not of class $C^{1}$, by proposition \ref{Prop_diff=c1},
there is a sequence $q_{n}\in H^{+}\left(  S\right)  $ for which
\[
q_{n}\rightarrow\gamma\left(  t^{p}\right)  ,N\left(  q_{n}\right)  \geq2.
\]
Thus there are at least two different past directed generators $\gamma_{n}%
^{1},~\gamma_{n}^{2}$ starting at $q_{n}$. It is easy to see that $\gamma
_{n}^{i},~i=1,2$ must converge\footnote{See chapter 3.3 in \cite{B-E_E} on the
limit of nonspacelike curves.} to $\overrightarrow{\gamma_{p}}\overset{def}{=}%
\gamma|_{\left[  t^{p},\beta\right)  }$, for example there is subsequence of
$\gamma_{n}^{1}$ which locally converges to a past directed light-like
geodesic starting at $\gamma\left(  t^{p}\right)  $ which must lie on
$H^{+}\left(  S\right)  $, but this is unique, as $N\left(  \gamma\left(
t^{p}\right)  \right)  =1$. Therefore, if $p_{n}^{i}\overset{def}{=}\gamma
_{n}^{i}\cap R$ then%
\begin{equation}
p_{n}^{1},~p_{n}^{2}\rightarrow\gamma\left(  t^{-}\right)  \label{E3}%
\end{equation}
as well.

Since $\partial\left(  J^{-}\left(  \gamma\left(  t^{p}\right)  \right)  \cap
R\right)  $ is smooth in $R\approx\mathbb{R}^{n-1}$, let $o$ be a point in the
interior of $J^{-}\left(  \gamma\left(  t^{p}\right)  \right)  \cap R$ such
that the line $\overline{o\gamma\left(  t^{-}\right)  }$ is normal to
$\partial\left(  J^{-}\left(  \gamma\left(  t^{p}\right)  \right)  \cap
R\right)  $ at $\gamma\left(  t^{-}\right)  ,$ i.e. $o$ is on the normal line
of $\partial\left(  J^{-}\left(  \gamma\left(  t^{p}\right)  \right)  \cap
R\right)  $ at $\gamma\left(  t^{-}\right)  $. As $q_{n}\rightarrow
\gamma\left(  t^{p}\right)  $ we have that
\begin{equation}
J^{-}\left(  q_{n}\right)  \cap R\rightarrow J^{-}\left(  \gamma\left(
t^{p}\right)  \right)  \cap R, \label{E4}%
\end{equation}
For step $5.$ consider the $2$-dimensional plane $P_{n}$ spanned by
$o,p_{n}^{1},p_{n}^{2}$ \ which intersects $\partial\left(  J^{-}\left(
q_{n}\right)  \cap R\right)  $ in a smooth curve and $p_{n}^{1},p_{n}^{2}$ are
on this curve\footnote{If $o,p_{n}^{1},p_{n}^{2}$ are collinear for every
$n>N$ for an $N\in\mathbb{N},$ then we can change $o$ to an other point on the
normal line,}. By $\left(  \ref{E3}\right)  $ a suitable subsequent of $P_{n}$
will converge to a normal plane $P$ of $\partial\left(  J^{-}\left(
\gamma\left(  t^{p}\right)  \right)  \cap R\right)  $ at $\gamma\left(
t^{-}\right)  $.

The continuity of the classical $2$-dimensional curvatures, the convergence
(\ref{E4}) and property P2 yield that $\partial\left(  J^{-}\left(
q_{n}\right)  \cap R\right)  $ has $2$-dimensional curvatures bigger than
$\kappa+1$ everywhere, if $n$ is big enough. Consider the plane curve
$\partial\left(  J^{-}\left(  q_{n}\right)  \cap R\right)  \cap P_{n}$ and the
(shortest) part of it between its points $p_{n}^{1}$ and $p_{n}^{2}$, see
\textsc{Figure \ref{f2}}, this curve segment will be denoted by $\phi_{n}$. By
lemma \ref{LMeusnier}, we have that the curvature of $\phi_{n}$ is everywhere
at least that of the minimal $2$-dimensional curvature of $\partial\left(
J^{-}\left(  q_{n}\right)  \cap R\right)  $. By property $P2$ and by the
convergence (\ref{E4}) we have that the curvature of $\phi_{n}$ is at least
$\kappa+1$ at every point of $\phi_{n}$ if $n$ is big enough. This proved step
$6.$%

\begin{figure}[ptb]%
\centering
\includegraphics[
height=4.2241cm,
width=12.8722cm
]%
{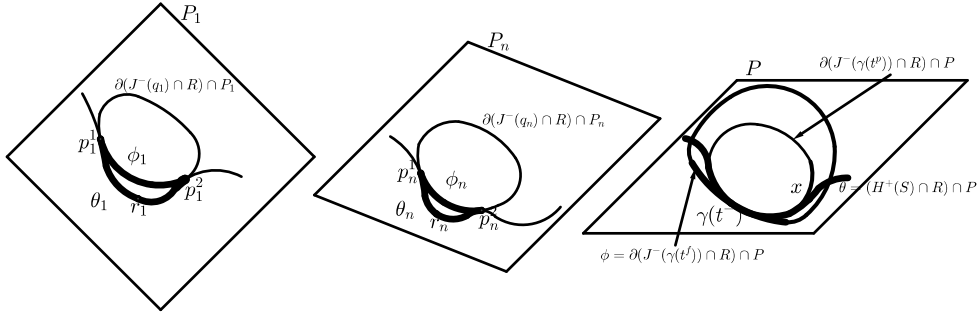}%
\caption{The convergence of "big curvature" points $r_{n}$ give a horizon
point $x$ in $I^{-}\left(  \gamma\left(  t^{f}\right)  \right)  $.}%
\label{f2}%
\end{figure}

Let $\theta_{n}$ be the segment of the curve $\left(  H^{+}\left(  S\right)
\cap R\right)  \cap P_{n}$ between its points $p_{n}^{1}$ and $p_{n}^{2}$
which is twice differentiable if $n$ is big enough, because $H^{+}\left(
S\right)  $ is at least of class $C^{2}$ at $\gamma\left(  t^{-}\right)  $ and
we have (\ref{E3}).

Note that the manifolds $\partial\left(  J^{-}\left(  \gamma\left(
t^{p}\right)  \right)  \cap R\right)  $ and $H^{+}\left(  S\right)  \cap
R=\partial\left(  \overline{D^{+}\left(  S\right)  }\cap R\right)  $ are
tangent at $\gamma\left(  t^{-}\right)  ,$ since $J^{-}\left(  \gamma\left(
t^{p}\right)  \right)  \cap R\subset\overline{D^{+}\left(  S\right)  }\cap R$
and $\gamma\left(  t^{-}\right)  $ is on their boundary an both boundary is at
least of class $C^{1}$. Since the line $\overline{o\gamma\left(  t^{-}\right)
}$ is transversal to the manifolds $\partial\left(  J^{-}\left(  \gamma\left(
t^{p}\right)  \right)  \cap R\right)  $ and $H^{+}\left(  S\right)  \cap R$,
the convergence (\ref{E4}) and a simple continuity argument yield that
$\theta_{n}$ and $\phi_{n}$ can be considered as function graphs in $P_{n}$.
Therefore we can apply lemma \ref{LAgetcurvature} in the appendix to prove
step $7$. Let $r_{n}\in\theta_{n}$ be the point where the formal curvature is
at least the minimal curvature of $\phi_{n}$, see definition \ref{Dref} in the
appendix. Thus, by step $6.$ we have that the curvature of $\theta_{n}$ at
$r_{n}$ is at least $\kappa+1$.

Step $8.$ Since $H^{+}\left(  S\right)  $ is at least of class $C^{2}$ at
$\gamma\left(  t^{-}\right)  $ and $P_{n}\rightarrow P$ we have that the
formal curvature expressions of the curves $\theta_{n}$ at $r_{n}$ will
converge to the curvature of the curve%
\[
\theta\overset{def}{=}\left(  H^{+}\left(  S\right)  \cap R\right)  \cap P
\]
at $\gamma\left(  t^{-}\right)  $. \ This gives that the formal curvature of
$\theta$ at $\gamma\left(  t^{-}\right)  $ is at least $\kappa+1$.

What remains is step $9$. By property P2 we have that the curvature of the
curve $\phi\overset{def}{=}\partial\left(  J^{-}\left(  \gamma\left(
t^{f}\right)  \right)  \cap R\right)  \cap P$ at $\gamma\left(  t^{-}\right)
$ is at most $\kappa$. Now in the plane $P$, we can apply remark
\ref{LAbigcurvature} to $\theta$ and $\phi$ at $\gamma\left(  t^{-}\right)  $
to get a point $x\in\theta\subset H^{+}\left(  S\right)  $ which is "above"
$\phi$, i.e. it must lie in the interior\footnote{We calculated all the
curvatures with respect to the inward directions of the "spheres". Since all
the curvatures are positive, this means that the curves bend to the inward
direction.} of $\left(  J^{-}\left(  \gamma\left(  t^{f}\right)  \right)  \cap
R\right)  \cap P$ which is $\left(  I^{-}\left(  \gamma\left(  t^{f}\right)
\right)  \cap R\right)  \cap P$, because $U\approx\mathbb{R}^{n}$ in step $1$.
was suitably small. But \ then $x$ and $\gamma\left(  t^{f}\right)  $ are
chronological related points on the horizon which contradicts the achronality
of $H^{+}\left(  S\right)  $.
\end{proof}

\begin{corollary}
If $\gamma:[(\alpha,\beta)\rightarrow H$ is a generator and $H$ which is of
class $C^{k}$, $k\geq2$ at an interior point $\gamma\left(  t\right)
,~t\in\left(  \alpha,\beta\right)  $, then $H$ is of class $C^{k}$ at every
point $\gamma\left(  t\right)  ,~t\in\left(  \alpha,\beta\right)  $.
\end{corollary}

\subsection{Appendix}

In the appendix we collect some technical lemmas. Since their proof is
straightforward we omit some calculations.

Let $M\subset\mathbb{R}^{m}$ be a $1$-codimensional smooth submanifold $p\in
M$, $v\in T_{p}M$, and $n\perp T_{p}M$ a normal vector of $M.$ Let $S\subset
M$ be the $2$-dimensional affine plane at $p$ parallel to $n$ and $v $.
Consider the curve $\theta\overset{def}{=}M\cap S$ as a plane curve in $S $.
Let
\[
\kappa\left(  v,M,n\right)
\]
denote the signed curvature of $\theta$ with respect to the base $v,~n$, which
is called \textbf{the normal curvature in the direction of }$v$\textbf{\ with
respect to }$n$\textbf{. }Note that $n$ show the direction in which the curves
with positive curvature bend.

\begin{lemma}
\label{LMeusnier}Let $M\subset\mathbb{R}^{m}$ be a $1$-codimensional smooth
submanifold and $P\subset\mathbb{R}^{n}$ a $2$-dimensional affine plane.
Assume that $\phi:\left(  -1,1\right)  \rightarrow M\cap P$ is a smooth
parameterization of a small part of the intersection curve $M\cap P$ and
$n\perp T_{\phi\left(  0\right)  }M$ is a normal vector for which
$\kappa\left(  \phi^{\prime}\left(  0\right)  ,M,n\right)  \geq0$. Then the
signed curvature of $\phi\subset P\approx\mathbb{R}^{2}$ as a plane curve at
$\phi^{\prime}\left(  0\right)  $ is at least $\kappa\left(  \phi^{\prime
}\left(  0\right)  ,M,n\right)  $. with respect to the base $\phi^{\prime
}\left(  0\right)  .~n^{p}$ where $n^{p}$ is the parallel part of $n$ to $P$.
\end{lemma}

The last line gives only that the signed curvature of $\phi$ in $P$ is also positive.

\begin{proof}
(Sketch) Consider the $3$-dimensional affine subspace $N\subset\mathbb{R}^{m}$
spanned by $P$ and the normal direction of $M$ at $\phi\left(  0\right)  $.
Since $N\cap M$ is locally a smooth submanifold in $N\approx\mathbb{R}^{3}$ at
$\phi\left(  0\right)  $ and all our curves lie in $N$, we can apply the
standard Meusnier theorem to $N\cap M$ in $N\approx\mathbb{R}^{3}$ for our curves.
\end{proof}

The curvature of a plane curve is usually defined for $C^{2}$ curves, however
we need it for $C^{1}$ curves which second derivatives exists everywhere, but
it is not necessarily continuous. But we can define a formal signed curvature
also for such curves which has the same geometric properties as the standard
one. For the sake of simplicity we will use function graphs.

\begin{definition}
\label{Dref}Let $f:I\rightarrow\mathbb{R}$ be a function on the interval $I$
which is twice differentiable at every point of $I$. The formal curvature of
the function graph at $t_{0}\in I$ is%
\[
\kappa_{f}\left(  t_{0}\right)  \overset{def}{=}\frac{f^{\prime\prime}\left(
t_{0}\right)  }{\left(  1+f^{\prime}\left(  t_{0}\right)  ^{2}\right)  ^{3/2}%
}.
\]

\end{definition}

The above definition is the usual one, but note that $\kappa\left(  t\right)
$ is not necessarily continuous. There is a geometric interpretation of the
curvature in the $C^{2}$ case which works also in this weaker case.

\begin{definition}
\label{Dcirc}Let $f:I\rightarrow\mathbb{R}$ be a function on the interval $I $
which is twice differentiable at every point of $I$. Let $r\left(
t_{0},t\right)  $ be the signed radius of the circle which is tangent to the
function graph at $t_{0}$ and contains the point $\left(  t,f\left(  t\right)
\right)  $ where the sign is $+$ if the center of the circle is in the
direction $\left(  -f^{\prime}\left(  t_{0}\right)  ,1\right)  $ from $\left(
t_{0},f\left(  t_{0}\right)  \right)  $, see \textsc{figure \ref{F3}.}
\end{definition}

%

\begin{figure}[ptb]%
\centering
\includegraphics[
height=5.3114cm,
width=16.9491cm
]%
{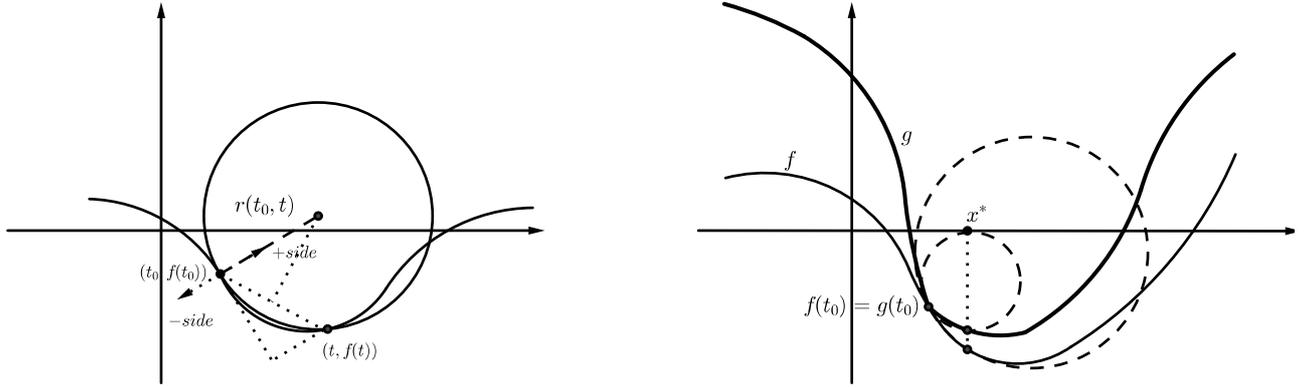}%
\caption{Left side: the circle in definition \ref{Dcirc} and the signs; Right
side: the situation in remark \ref{LAbigcurvature}}%
\label{F3}%
\end{figure}

\begin{lemma}
Let $f:I\rightarrow\mathbb{R}$ be a function on the interval $I$ which is
twice differentiable at every point of $I$. Then%
\[
\lim_{t\rightarrow t_{0}}\frac{1}{r\left(  t_{0},t\right)  }=\kappa_{f}\left(
t_{0}\right)  \,\ \forall t_{0}\in I.
\]
Thus the limit exists and is equal to the formal curvature.
\end{lemma}

\begin{proof}
By standard geometrical methods indicated on \textsc{figure \ref{F3},} we can
calculate
\[
r\left(  t_{0},t\right)  =\frac{2\left\langle \left(  -f^{\prime}\left(
t_{0}\right)  ,1\right)  ,\left(  \left(  t-t_{0}\right)  ,\left(  f\left(
t\right)  -f\left(  t_{0}\right)  \right)  \right)  \right\rangle \frac
{1}{\sqrt{1+f^{\prime}\left(  t_{0}\right)  ^{2}}}}{\left(  t-t_{0}\right)
^{2}+\left(  f\left(  t\right)  -f\left(  t_{0}\right)  \right)  ^{2}},
\]
where $\left\langle .,.\right\rangle $ is the usual inner product. Since the
L'Hospital rule can be applied twice to $\lim_{t\rightarrow t_{0}}\frac
{1}{r\left(  t_{0},t\right)  }$ a standard calculation shows, that the limit
exists and is equal to $\kappa_{f}\left(  t_{0}\right)  $.
\end{proof}

Now if two function graphs are tangent at some point, the above geometric
interpretation of the curvature yields:

\begin{remark}
\label{LAbigcurvature}Let $f,~g:I\rightarrow\mathbb{R}$ be functions on the
interval $I$ which are twice differentiable at every point of $I$ and tangent
at $t_{0}\in I$. If $\kappa_{f}\left(  t_{0}\right)  <\kappa_{g}\left(
t_{0}\right)  $ then there is a $t^{\ast}\in I$ close enough to $t_{0}$ for
which $f\left(  t^{\ast}\right)  <g\left(  t^{\ast}\right)  $, see
\textsc{figure \ref{F3}.}
\end{remark}

\begin{lemma}
\label{LAgetcurvature}Let $f,~g:\left[  \alpha,\beta\right]  \rightarrow
\mathbb{R}$ be functions which are twice differentiable at every point of
$\left[  \alpha,\beta\right]  $ and their function graphs are tangent at
$\alpha$ and $\beta$. Assume that $g\left(  t\right)  \leq f\left(  t\right)
,~\forall t\in\left[  \alpha,\beta\right]  $. Then there is a $t^{\ast}%
\in\left[  \alpha,\beta\right]  $ for which $\min_{t\in\left[  \alpha
,\beta\right]  }\kappa_{f}\left(  t\right)  \leq\kappa_{g}\left(  t^{\ast
}\right)  $.
\end{lemma}

\begin{proof}
Let $0<C$ be the biggest value for which the function graphs of $f-C$ and $g$
have a common point. At this point the graphs must be tangent. Let $t^{\ast} $
be the parameter at which they are tangent. Since $f\left(  t\right)  -C\leq
g\left(  t\right)  $ locally at $t^{\ast}$. We must have $\kappa_{f}\left(
t^{\ast}\right)  =\kappa_{f-C}\left(  t^{\ast}\right)  \leq\kappa_{g}\left(
t^{\ast}\right)  $ by remark \ref{LAbigcurvature}.
\end{proof}

\begin{acknowledgement}
This work was supported by NKFI- (OTKA) Grant nr. K-128862 and Application
Domain Specific Highly Reliable IT Solutions\textquotedblright\ project has
been implemented with the support provided from the National Research,
Development and Innovation Fund of Hungary, financed under the Thematic
Excellence Programme TKP2020-NKA-06 (National Challenges Subprogramme) funding scheme.
\end{acknowledgement}

\end{document}